\documentclass[a4paper]{amsart}

\usepackage[utf8]{inputenc}
\usepackage{amsmath, amsthm, amssymb, amscd}
\usepackage{hyperref}

\usepackage{tikz}
\usepackage{tikz-cd}
\usetikzlibrary{matrix,arrows}

\theoremstyle{plain}
\newtheorem{theorem}{Theorem}[section]
\newtheorem{lemma}[theorem]{Lemma}
\newtheorem{corollary}[theorem]{Corollary}
\newtheorem{proposition}[theorem]{Proposition}

\theoremstyle{definition}
\newtheorem{definition}[theorem]{Definition}
\newtheorem{example}[theorem]{Example}

\theoremstyle{remark}
\newtheorem{remark}[theorem]{Remark}

\DeclareMathOperator{\Mod}{Mod}
\DeclareMathOperator{\Vecc}{Vec}
\DeclareMathOperator{\Vecf}{Vecf}

\DeclareMathOperator{\Triv}{Triv}

\DeclareMathOperator{\Rep}{Rep}
\DeclareMathOperator{\Repf}{Repf}

\DeclareMathOperator{\Ind}{Ind}

\DeclareMathOperator{\ID}{ID}

\DeclareMathOperator{\Spec}{Spec}
\DeclareMathOperator{\Sym}{Sym}
\DeclareMathOperator{\Aut}{Aut}
\DeclareMathOperator{\Isom}{Isom}

\DeclareMathOperator{\Hom}{Hom}
\DeclareMathOperator{\End}{End}

\DeclareMathOperator{\GL}{GL}

\title{A general theory of Andr\'e's solution algebras}
\author{Levente Nagy and Tamás Szamuely}
\date{}

\begin{document}

\maketitle

\begin{abstract}
We extend Yves Andr\'e's theory of solution algebras in differential Galois theory to a general Tannakian context. As applications, we establish analogues of his correspondence between solution fields and observable subgroups of the Galois group for iterated differential equations in positive characteristic and for difference equations. The use of solution algebras in the difference algebraic context also allows a new approach to recent results of Philippon and Adamczewski--Faverjon in transcendence theory.
\end{abstract}

\section{Introduction}

In classical differential Galois theory one considers a linear differential equation over a field $K$ equipped with a nontrivial derivation $\partial$ such that $k:=\ker(\partial)$ is algebraically closed of characteristic 0. Following Kolchin, one constructs a differential field extension $L|K$, called the Picard--Vessiot extension, which is generated over $K$ by all solutions of the equation and their derivatives. The group $G$ of relative automorphisms of $L|K$ respecting the derivation has a natural structure of a linear algebraic group over the constant field  $k\subset K$, and there is a Galois correspondence between intermediate differential fields of $L|K$ and closed subgroups of $G$. For these classical results we refer to the book of van der Put and Singer \cite{MR1960772}.

In his recent paper \cite{MR3215927}, Yves Andr\'e introduced a refinement of the differential Galois correspondence by considering intermediate extensions generated by {\em some} but not necessarily all solutions of the differential equations. He called these subfields solution fields and showed that they correspond to observable subgroups of the differential Galois group, i.e. closed subgroups $H\subset G$ with quasi-affine quotient $G/H$. Using the more refined Tannakian approach to differential Galois theory, Andr\'e also showed that solution fields arise as fraction fields of so-called solution algebras which are generalizations of the classical Picard--Vessiot algebras, and established a correspondence between solution algebras and affine quasi-homogeneous varieties under the differential Galois group.

At the end of his paper (\cite{MR3215927}, Remark 6.5 (3)), Andr\'e writes that he expects a similar theory of solution algebras in characteristic $p>0$ using iterated derivations as well as a similar theory for difference equations. In this paper we confirm his expectations.

For the characteristic $p$ theory, we work with the iterative differential modules (or ID-modules) of Matzat and van der Put \cite{MR1978401} (see the beginning of Section 5 for a brief summary of the basic definitions). To an ID-module $\mathcal M$ defined over an ID-field $\mathcal K$ they associate a Picard--Vessiot  extension $\mathcal{J}{\mid}\mathcal{K}$ in the ID-setting and an associated Galois group scheme $G$. This group scheme satisfies an iterated differential Galois correspondence, extended to possibly inseparable Picard--Vessiot extensions by Maurischat \cite{MR2657686}. Mimicking Andr\'e's definition, we say that an extension of ID-fields $\mathcal{L}{\mid}\mathcal{K}$ is a {\em solution field} for $\mathcal{M}$ if the constant field of $\mathcal{L}$ is $k$ and there exists  a morphism of ID-modules $\mathcal{M} \to \mathcal{L}$ whose image generates the underlying field extension $L{\mid}K$. One of our main results is then the following generalization of a theorem of Andr\'e's:

\begin{theorem}[= Theorem \ref{idobs}]
An intermediate ID-extension $\mathcal{L}$ of $\mathcal{J}{\mid}\mathcal{K}$ is a solution field for $\mathcal{M}$ if and only if the corresponding subgroup scheme $H$ is an observable subgroup scheme of the Galois group scheme $G$.
\end{theorem}

Here a closed subgroup scheme $H\subset G$ is called observable if the quotient $G/H$ is a quasi-affine scheme. Note that, in contrast to Andr\'e's setting, we allow our group schemes to be non-reduced. In fact, in Section 6 we shall exhibit an example of a solution field corresponding to a non-reduced closed subgroup scheme of a reduced iterative differential Galois group which is moreover not a Picard--Vessiot extension.

The proof of the above result has two main inputs: the iterative differential Galois correspondence quoted above and a generalization of Andr\'e's theory of solution algebras. We show that it is possible to develop the latter in a general Tannakian setting without reference to differential algebra. First, in Sections 2 and 3 we describe a general theory of Picard--Vessiot objects in $k$-linear tensor categories satisfying certain natural conditions, including the existence of a non-neutral fibre functor playing the role of the forgetful functor for differential modules. In Theorem \ref{pvfibfun} we show that Picard--Vessiot objects correspond to neutral fibre functors on rigid abelian tensor subcategories $\langle X\rangle_\otimes$ generated by a single object $X$. This generalizes the correspondence of Deligne and Bertrand in the last section of \cite{MR1106898}.\footnote{Note that such a theory was also developed in \cite{MR3633711} (in fact, independently, of our work). Although there are some unavoidable similarities, we feel that our approach is simpler.} Afterwards, we give a general definition of solution algebras associated with objects $X$ in the tensor categorical setting (Definition \ref{defi-solution-algebra}), and prove:

\begin{theorem} Let $\langle X\rangle_\otimes$ be a full rigid abelian tensor subcategory in a $k$-linear tensor category as above, equipped with a neutral fibre functor $\omega$. Given a solution algebra $\mathcal S$ associated with $X$, its image $\omega(\mathcal S)$ is a finitely generated $k$-algebra whose spectrum carries an action of the Tannakian fundamental group $G$ of $\omega$. It is moreover a quasi-homogeneous $G$-scheme, i.e. it has a schematically dense $G$-orbit.

The assignment $\mathcal{S}\mapsto\Spec (\omega(\mathcal S))$ induces an anti-equivalence between the category of solution algebras and that of affine quasi-homogeneous $G$-schemes of finite type over $k$.
\end{theorem}

This will be proven in a somewhat stronger form in Theorem \ref{proposition-solution-algebra-equiavalence}, thereby giving an abstract form of another theorem of Andr\'e \cite{MR3215927}.

Our abstract formulation applies in other situations as well. For instance, in the last section we briefly sketch how to apply it to difference modules to obtain a refinement of the Galois correspondence for difference equations. This context is particularly interesting because of applications in transcendence theory: as Yves Andr\'e pointed out to us, solution algebras for difference modules can be used to reprove results of Adamczewski--Faverjon \cite{adamczewski} and Philippon \cite{philippon} on the specialization of algebraic relations among Mahler functions (see Corollary \ref{mahlercor}).

The main results of this paper come from the 2018 Central European University doctoral thesis \cite{thesis} of the first author. We thank Yves Andr\'e for his enlightening comments on a preliminary version, and in particular for suggesting Remark \ref{pvremas} (2) and Corollary \ref{mahlercor}. We are also grateful to the referee for several pertinent suggestions.

\section{Solvable objects in tensor categories}

In this section and the next one we develop an abstract version of the Tannakian approach to the theory of Picard--Vessiot extensions that will be applied in the concrete situation of iterative differential modules and difference modules in subsequent sections.

\par We begin by some generalities concerning tensor categories; see e.g. \cite{delignemilne} as a basic reference. In this text by a \emph{tensor category} $\mathcal{C}$ we shall mean a $k$-linear abelian symmetric monoidal category over a field $k$. We assume that $\mathcal{C}$ admits small colimits that commute with the tensor product in both variables. Tensor functors will be assumed to be $k$-linear and preserving small colimits.

\par The unit object $1$ is a ring object in the tensor category $\mathcal C$. {\em We shall assume throughout that this ring object is  simple}, i.e. every nontrivial morphism from $1$ to a module object over it is a monomorphism. (Note that in ${\rm Mod}(R)$ the unit object $R$ is simple if and only if $R$ is a simple ring.) In this case a Schur Lemma type argument yields that the endomorphism ring $\End_{\mathcal C}(1)$ is a field $k$; we shall assume that $\mathcal C$ is $k$-linear with respect to $k=\End_{\mathcal C}(1)$.

\par We shall also consider more general ring objects $\mathcal A$ in $\mathcal C$ (always assumed to be commutative) and module objects over them, forming a subcategory $\Mod_{\mathcal C}(\mathcal{A})$ in $\mathcal C$. Morphisms between two objects $X$ and $Y$ of $\Mod_{\mathcal C}(\mathcal{A})$ will be denoted by $\Hom_{\mathcal{A}}(X,Y)$; for $X=Y$ we use the notation $\End_{\mathcal{A}}(X)$. For generalities on ring objects, see e.g. \cite{MR1012168}, \S 5.

We first give an abstract version of the notion of trivial differential modules. This is enabled by the following proposition.

\begin{proposition}
Let $\mathcal{C}$ be a tensor category over the field $k$. Up to unique isomorphism there is a unique tensor functor
\[
\tau \colon \Mod(k) \to \mathcal{C}.
\]
\end{proposition}

\begin{proof} According to (\cite{MR3144607} Proposition 2.2.3), given a commutative $k$-algebra $A$ and an arbitrary tensor category $\mathcal C$ (not necessarily satisfying $\End_{\mathcal{C}}(1)\cong k$) there is an equivalence of categories
between the category of $k$-linear tensor functors  $\Mod(A)\to\mathcal{C}$ and that of $k$-algebra homomorphisms $A\to\End_{\mathcal{C}}(1)$.
But for $k=A=\End_{\mathcal{C}}(1)$  there is only one $k$-algebra homomorphism $k\to\End_{\mathcal{C}}(1)$, namely the identity morphism.
\end{proof}

Since $\tau$ commutes with small direct limits by assumption, it has the following  `coordinatization': given $M\in\Mod(k)$, choose a basis to write it as a direct sum $\bigoplus k$ of copies of $k$. Then $\tau(M)=\bigoplus 1$ with the same index set. Morphisms have a similar description via infinite matrices.

\par An object $X$ of $\mathcal{C}$ will be called \emph{trivial} if $X$ is isomorphic to an object of the form $\tau(M)$ for a $k$-module $M$. The functor $\tau$ will be called the \emph{trivial object functor} of $\mathcal{C}$ over $k$. The category $\Triv(\mathcal{C})$ of trivial objects of $\mathcal{C}$ is defined as the full subcategory of $\mathcal{C}$ spanned by trivial objects.

\begin{proposition}\label{proposition-subquotient-of-trivial-is-trivial} The functor $\tau$
induces an equivalence of tensor categories between $\Mod(k)$ and $\Triv(\mathcal{C})$,
with a quasi-inverse given by the restriction of the functor
\[
(-)^{\nabla} := \Hom_{\mathcal{C}}(1,-) \colon \mathcal{C} \to \Mod(k)
\]
to $\Triv(\mathcal{C})$.

Thus the tensor category $\Triv(\mathcal{C})$ is abelian, and the restriction of the functor $\nabla$ to $\Triv(\mathcal{C})$ is a faithful exact tensor functor.
\end{proposition}

\begin{proof}
To check that $\nabla$ is a quasi-inverse to $\tau$ we may reduce via `coordinatization' to the full subcategories spanned by the unit objects, where it is clear. Exactness of $\nabla$ on $\Triv(\mathcal{C})$ means by simpleness of $1$ that every epimorphism $X\to 1$ splits in $\Triv(\mathcal{C})$, which follows by writing $X$ as a direct sum of copies of $1$ and using simpleness again.
\end{proof}

We can use this category equivalence for detecting dualizable objects (in the sense of tensor categories) among trivial objects.

\begin{corollary}\label{proposition-dualizable-trivial-objects} A trivial object $X$ is dualizable  if and only if $X^{\nabla}$ is a finite dimensional $k$-vector space. In this case its dual is again a trivial object.
\end{corollary}

\begin{proof} In view of the proposition, the first statement follows from the fact that the dualizable objects in $\Mod(k)$ are the finite-dimensional vector spaces. As for the second, the dual of a trivial object $X$  is isomorphic to $\tau((X^{\nabla})^{\vee})$, where ${}^{\vee}$ denotes the $k$-dual of a vector space.
\end{proof}

One drawback of the results so far is that the subcategory $\Triv(\mathcal{C})$ may not be a {\em fully} abelian category of $\mathcal C$, i.e. it is not clear that the kernel and a cokernel of a morphism between trivial objects are the same in $\Triv(\mathcal{C})$ and $\mathcal C$. To remedy this, we impose an extra condition.

\begin{definition}
A tensor category $\mathcal C$ is {\em pointed} if there exits a faithful exact tensor functor $\vartheta:\, {\mathcal C}\to\Mod(R)$ for some ring $R$.
\end{definition}

Main examples to bear in mind are categories of differential and difference modules as well as representation categories of affine group schemes. In these examples the faithful exact tensor functor $\vartheta$ is just the forgetful functor to the appropriate module category.

\begin{proposition}\label{fullyab}
If $\mathcal C$ is pointed, the inclusion functor $\Triv(\mathcal{C})\to\mathcal C$ is exact, and hence $\Triv(\mathcal{C})$ is a fully abelian tensor subcategory of $\mathcal C$.
\end{proposition}

\begin{proof}
By (\cite{MR3144607}, Corollary 2.2.4), the composite functor
\[
\vartheta \circ \tau \colon \Mod(k) \to \Mod(R)
\]
is isomorphic to the base change functor induced by a $k$-morphism $k \to R$. But there is only one such morphism, the structure morphism which is moreover faithfully flat as $k$ is a field by our assumption. It follows that the pullback functor $\sigma^* \cong \vartheta \circ \tau$ is faithful exact, hence $\tau$ is also faithful exact.
\end{proof}

We now come to an abstract version of the notion of solvability for differential modules.

\begin{definition}
Let $\mathcal{C}$ be a tensor category, $X$  an object of $\mathcal{C}$ and $\mathcal{A}$  a ring object in $\mathcal{C}$. We say that the object $X$ is solvable in $\mathcal{A}$ if the $\mathcal{A}$-module $\mathcal{A} \otimes X$ is a trivial object in $\Mod_{\mathcal{C}}(\mathcal{A})$.
\end{definition}

Recall that a ring $\mathcal{A}$ in $\mathcal C$ is {\em flat} (resp. {\em faithfully flat}) (over the unit ring $1$) if the base change functor
$\mathcal{A} \otimes - \colon \mathcal{C} \to \Mod_{\mathcal{C}}(\mathcal{A})$
is an exact (resp. a faithful exact) functor.

\begin{corollary}\label{solv}
Let $\mathcal C$ be a pointed tensor category, and $\mathcal A$ a faithfully flat simple ring in $\mathcal C$. The full subcategory of $\mathcal C$ spanned by objects solvable in $\mathcal A$ is a fully abelian subcategory closed under arbitrary small colimits and tensor products. The unit object of $\mathcal{C}$ is solvable in $\mathcal{A}$.
\end{corollary}

\begin{proof}
This follows from Propositions \ref{proposition-subquotient-of-trivial-is-trivial} and \ref{fullyab} since base change by $\mathcal A$ is an exact faithful functor by assumption.
\end{proof}

As for dualizable objects, we have:

\begin{proposition}\label{proposition-dualizable-solvable-objects} In the situation of the previous corollary an $\mathcal A$-solvable object $X$ is dualizable in $\mathcal C$ if and only if $\mathcal{A} \otimes X$ is dualizable in $\Mod_{\mathcal{C}}(\mathcal{A})$.

If moreover $\mathcal A$ is a simple ring, then an $\mathcal{A}$-solvable object $X$ is dualizable if and only if the $\End_{\mathcal{A}}(\mathcal{A})$-vector space $\Hom_{\mathcal{A}}(\mathcal{A}, \mathcal{A}\otimes X)$ is finite dimensional. In this case the dual of $X$ is also $\mathcal{A}$-solvable.
\end{proposition}

Note that if $\mathcal A$ is simple, the ring $\End_{\mathcal{A}}(\mathcal{ A})$ is a field (it is commutative as $\mathcal A$ is the unit object of $\Mod_{\mathcal{C}}(\mathcal{A})$, and a skew field by a Schur Lemma type argument).

\begin{proof} This is an application of faithfully flat descent in tensor categories (see \cite{MR1106898}, 4.1--4.2 and \cite{MR1012168}, \S 5). Recall that a descent datum on an $\mathcal{A}$-module $\mathcal{M}$ is an isomorphism $\mathcal{A} \otimes \mathcal{M} \cong \mathcal{M} \otimes \mathcal{A}$ of $\mathcal{A} \otimes \mathcal{A}$-modules that satisfies the cocycle condition. As in the classical case, any $\mathcal M$ obtained via base change by $\mathcal A$ carries a descent datum and conversely, the descent datum is effective for faithfully flat $A$.

Assume now $X$ is an object such that $\mathcal{M}:=\mathcal{A} \otimes X$ is dualizable with dual $\mathcal N$. We construct a descent datum on $\mathcal{N}$ as follows. The $\mathcal{A} \otimes \mathcal{A}$-module $\mathcal{A} \otimes \mathcal{M}$ is isomorphic to $(\mathcal{A} \otimes \mathcal{A}) \otimes_{\mathcal{A}} \mathcal{M} $, hence it has a dual in the category of $\mathcal{A} \otimes \mathcal{A}$-modules, namely $\mathcal{A} \otimes \mathcal{N} \cong (\mathcal{A} \otimes \mathcal{A}) \otimes_{\mathcal{A}} \mathcal{N}$. The same can be said of $\mathcal{M} \otimes \mathcal{A}$. Since the dual in a tensor category is uniquely determined, we can dualize the isomorphism giving the descent datum on $\mathcal{M}$ and obtain an isomorphism
\[
\mathcal{A} \otimes \mathcal{N} \cong (\mathcal{A} \otimes \mathcal{A}) \otimes_{\mathcal{A}} \mathcal{N} \cong \mathcal{N} \otimes_{\mathcal{A}} (\mathcal{A} \otimes \mathcal{A}) \cong \mathcal{N} \otimes \mathcal{A}.
\]
By a similar argument we deduce a cocycle condition for $\mathcal N$ from that on $\mathcal M$ and conclude by faithfully flat descent that $\mathcal N$ is of the form $\mathcal{A} \otimes Y$ for an object $Y$, where $Y$ satisfies the axioms for a dual of $X$ in the tensor category $\mathcal C$. This proves the first statement, and the second one follows from Corollary \ref{proposition-dualizable-trivial-objects}, again via descent.
\end{proof}

\section{Abstract Picard-Vessiot theory}

We now come to the first key definition in this paper. In the case of differential modules over a differential ring it specializes to the notion of Picard-Vessiot rings as defined in (\cite{MR1862024}, \S 3.4) and (\cite{MR3215927}, \S 2.4).

\begin{definition}\label{definition-p-v-ring}
Let $\mathcal{C}$ be a tensor category, and $X$ a dualizable object of $\mathcal{C}$. A ring $\mathcal{P}$ in $\mathcal{C}$ is called a Picard-Vessiot ring for $X$ in $\mathcal{C}$ if it satisfies the following properties:
\begin{enumerate}
\item $\mathcal{P}$ is a faithfully flat simple ring in $\mathcal{C}$,
\item the homomorphism $k = \End_{\mathcal{C}}(1) \to \End_{\mathcal{P}}(\mathcal{P})$, induced by the morphism $1 \to \mathcal{P}$, is an isomorphism,
\item the object $X$ is solvable in $\mathcal{P}$,
\item the ring $\mathcal{P}$ is minimal with these properties, i.e. if $\mathcal{P}'$ is another ring in $\mathcal{C}$ satisfying the previous properties and $\mathcal{P}' \to \mathcal{P}$ is a ring homomorphism, then it is an isomorphism.
\end{enumerate}
\end{definition}

For a dualizable object $X$ of $\mathcal{C}$ we shall denote by $\langle X \rangle_{\otimes}$ the full essential subcategory of $\mathcal{C}$ consisting of subquotients of finite direct sums of objects of the form $X^{\otimes i} \otimes (X^{\vee})^{\otimes j}$. The category of finite dimensional vector spaces over a field $k$ will be denoted by $\Vecf(k)$.

\begin{proposition}\label{proposition-p-v-ring-induces-fibre-functor}
Assume that the tensor category $\mathcal C$ is pointed, and  there exists a Picard-Vessiot ring $\mathcal{P}$ for the dualizable object $X$ in $\mathcal{C}$. Then
\[
\omega_{\mathcal{P}} = \Hom_{\mathcal{P}}(\mathcal{P}, \mathcal{P} \otimes -) \colon \langle X \rangle_{\otimes} \to \Vecf(k)
\]
is a $k$-linear faithful exact tensor functor, and
 $\langle X \rangle_{\otimes}$ equipped with $\omega_{\mathcal{P}}$ is a neutral Tannakian category over $k$.
\end{proposition}

\begin{proof}
Firstly, by Corollary \ref{solv} and Proposition \ref{proposition-dualizable-solvable-objects} every object of $\langle X \rangle_{\otimes}$ is solvable. Base change by $\mathcal P$ is fully faithful and maps objects of $\langle X \rangle_{\otimes}$ to the subcategory of trivial objects of $\Mod_{\mathcal{C}}(\mathcal{P})$. On the latter category the functor $\Hom_{\mathcal{P}}(\mathcal{P}, -)$ is none but the functor $\nabla$, hence by Proposition \ref{proposition-subquotient-of-trivial-is-trivial} the composite $\omega_{\mathcal{P}} = \Hom_{\mathcal{P}}(\mathcal{P}, -) \circ (\mathcal{P} \otimes -)$ is a faithful exact tensor functor. Finally, $\omega_{\mathcal{P}}$ has values in $\Vecf(k)$ by Proposition \ref{proposition-dualizable-solvable-objects}.
\end{proof}

The Tannakian fundamental group scheme $G =\underline{\Aut}^{\otimes}(\omega_{\mathcal{P}})$ will be called the \emph{Galois group scheme of $X$} (pointed in $\omega_{\mathcal{P}}$).

\begin{remark}\label{rewrite}\rm
In the situation of the proposition we have isomorphisms of $k$-vector spaces \[
\omega_{\mathcal{P}}(Y) =\Hom_{\mathcal{P}}(\mathcal{P}, \mathcal{P} \otimes Y) \cong \Hom_{\mathcal{C}}(1, \mathcal{P} \otimes Y)\cong \Hom_{\mathcal{C}}(Y^{\vee}, \mathcal{P}).
\]  for every $Y \in \langle X \rangle_{\otimes}$.
 In other words, the vector space $\omega_{\mathcal{P}}(Y)$  can be viewed as the vector space of "solutions" of the dual $Y^\vee$ of $Y$ in $\mathcal{P}$.
\end{remark}

We state now the converse of Proposition \ref{proposition-p-v-ring-induces-fibre-functor}.

\begin{theorem}\label{pvfibfun}
Let $X$ be a dualizable object of the pointed tensor category $\mathcal{C}$ such that $\langle X \rangle_{\otimes}$ is a rigid $k$-linear abelian tensor subcategory of $\mathcal{C}$. The map $\mathcal{P}\mapsto \omega_{\mathcal P}$ induces a bijective correspondence between isomorphism classes of Picard-Vessiot rings for $X$ in $\mathcal{C}$ and of $k$-valued fibre functors on $\langle X \rangle_{\otimes}$.
\end{theorem}

For the proof of the theorem we first examine the case of representation categories.

\begin{proposition}\label{proposition-regular-rep-is-p-v-ring}
Let $G$ be an affine group scheme of finite type over a field $k$. Consider the tensor category $\Rep_k(G)$ of  $k$-representations of $G$ viewed as a pointed tensor category via the forgetful functor $\omega$.

\par The regular representation $\mathcal{O}(G)$ is a Picard-Vessiot ring for the full subcategory $\Repf_k(G)$ of finite-dimensional representations in $\Rep_k(G)$, and the associated fibre functor is isomorphic to the restriction of $\omega$ to  $\Repf_k(G)$.
\par Conversely, every Picard-Vessiot ring for $\Repf_k(G)$ in $\Rep_k(G)$ with this property is isomorphic to the regular representation $\mathcal{O}(G)$.
\end{proposition}

\begin{proof}
The regular representation $\mathcal{O}(G)$ is a faithfully flat ring in $\Rep_k(G)$ since so is its image under the forgetful functor in $\Vecc(k)$. It is moreover a simple ring (indeed, the scheme $G$ equipped with its canonical $G$-action has no nontrivial proper closed $G$-subsets). The elements of the endomorphism ring of the regular $G$-representation $\mathcal{O}(G)$ can be identified with the $G$-invariant regular functions on $G$ and hence they are just the constants. This shows properties (1)-(2) of a Picard-Vessiot ring.

\par Let now $V$ be a finite dimensional representation, and denote by $V_\tau$ the underlying vector space of $V$ viewed as a trivial $G$-representation. Consider the associated vector bundles $\mathbb{V}=\Spec(\Sym^*(V^{\vee}))$ and $\mathbb{V_\tau}=\Spec(\Sym^*(V^{\vee}_\tau))$. Solvability of $V$ in $\mathcal{O}(G)$ is equivalent to the existence of a $G$-equivariant isomorphism of $G$-schemes
\begin{equation}\label{Vtriv}
_GG \times_k \mathbb{V}_{\rho} \to {}_GG \times_k \mathbb{V}_{\tau},
\end{equation}
where $_GG$ is the affine $G$-scheme associated with $\mathcal{O}(G)$.
Such an isomorphism is given on scheme-theoretic points by $(g,v) \mapsto (g,g^{-1}v)$, whence property (3).

\par Lastly, we have to show that the regular representation satisfies property (4). Let $\mathcal{P}$ be a ring in $\Rep_k(G)$ having the necessary properties and let $\lambda:\,\mathcal{P} \to \mathcal{O}(G)$ be a ring homomorphism in $\Rep_k(G)$. Since $\mathcal{P}$ is a simple ring, this homomorphism is injective, hence we only have to show that it is surjective. To see this, note that $\lambda$ induces a morphism
\[
\lambda^*:\,\omega_{\mathcal{P}} \cong \Hom_{G}(k, \mathcal{P} \otimes_k -) \to \Hom_{G}(k, \mathcal{O}(G) \otimes_k  -) \cong \omega
\]
between the associated fibre functors. But $\Repf_k(G)$ is a rigid tensor category, hence this morphism is in fact an isomorphism of tensor functors (see \cite{delignemilne}, Proposition 1.13). Consider a finite dimensional subrepresentation  $W\subset\mathcal{O}_G$, and substitute its dual $W^\vee$ in $\lambda^*$. As in Remark \ref{rewrite}, we may rewrite the result as an isomorphism
\[
\Hom_{G}(W, \mathcal{P}) \stackrel\sim\to \Hom_{G}(W, \mathcal{O}(G)).
\]
Consequently, the embedding $W \hookrightarrow \mathcal{O}(G)$ factors through the embedding $\mathcal{P}\hookrightarrow \mathcal{O}(G)$. As this holds for every $W$, we are done.

\par To see that the associated fibre functor of $\mathcal{O}(G)$ is the forgetful functor $\omega$, it suffices by construction  to take $G$-invariant elements in the isomorphism
\[
V_{\rho} \otimes_k \mathcal{O}(G) \cong V_{\tau} \otimes \mathcal{O}(G)
\]
deduced from (\ref{Vtriv}). Conversely, let $\mathcal{P}$ be a Picard-Vessiot ring for $\Repf_k(G)$ in $\Rep_k(G)$ whose associated fibre functor is $\omega$. We show that there exists a homomorphism $\mathcal{O}(G) \to \mathcal{P}$; by property (4) of Picard--Vessiot rings it must then be an isomorphism. As in the proof of property (4) for $\mathcal{O}(G)$ above, we can write $\mathcal{O}(G)$ as a colimit $\varinjlim W_i$ of finite-dimensional subrepresentations $W_i$, and deduce from the isomorphism of tensor functors $\omega_{\mathcal P}\cong\omega$ a compatible system of morphisms $W_i \to \mathcal{P}$. These assemble to the required homomorphism $\mathcal{O}(G) \to \mathcal{P}$.
\end{proof}

For the proof of Theorem \ref{pvfibfun} we also need a lemma concerning Ind-categories of tensor categories; we use (\cite{MR2182076}, Chapter 6) as our basic reference on this topic. Recall first from (\cite{MR2182076}, Corollary 6.3.2) that given a category $\mathcal C$ admitting small filtered colimits and a functor $F:\, {\mathcal T}\to \mathcal C$ from another category $\mathcal T$, the functor $F$ has a unique extension $JF:\, \Ind({\mathcal T})\to \mathcal C$.

\begin{lemma}\label{indlem} Let $\mathcal T$ be an abelian tensor category.
\begin{enumerate}
\item The category $\Ind({\mathcal T})$ is again an abelian tensor category.
\item Assume given a fully faithful exact tensor functor $F:\,\mathcal T\to \mathcal C$, where  $\mathcal C$ is an abelian tensor category in which small filtered colimits are exact. If for all objects $P$ of $\mathcal T$ the functor $\Hom_{\mathcal C}(F(P),\_\_)$ commutes with small filtered colimits of objects of the form $F(T)$ with $T\in {\mathcal T}$,  then the extension $JF$ is again a fully faithful and exact tensor functor. Consequently, it realizes $\Ind({\mathcal T})$ as a fully abelian tensor subcategory of $\mathcal C$.
\end{enumerate}
\end{lemma}

\begin{proof}
For statement (1), recall that by (\cite{MR2182076}, Theorem 8.6.5) the category $\Ind({\mathcal T})$ is abelian and admits small colimits. The tensor structure extends naturally to the Ind-category, and the extension $JF$ of the functor $F$ in (2) is again a tensor functor which is moreover exact by exactness of small filtered colimits in $\mathcal C$. Its fully faithfulness results from (\cite{MR2182076}, Proposition 6.3.4) and its proof, or \cite{MR0354652}, expos\'e I, Proposition 8.7.5 a).
\end{proof}

\begin{proof}[Proof of Theorem \ref{pvfibfun}] By Proposition \ref{proposition-p-v-ring-induces-fibre-functor} a Picard--Vessiot ring for $\langle X \rangle_{\otimes}$ induces a neutral fibre functor. Conversely, assume there exists such a neutral fibre functor $\omega$ on $\langle X \rangle_{\otimes}$. By the main theorem of neutral Tannakian categories (\cite{delignemilne}, Theorem 2.11) we have an equivalence of tensor categories $\langle X \rangle_{\otimes}\cong \Repf_k(G)$ for the associated Tannakian fundamental group scheme $G$, with $\omega$ inducing the forgetful functor on $\Repf_k(G)$. Moreover, by Lemma \ref{indlem} (1) the Ind-category $\Ind\langle X \rangle_{\otimes}$ is again an abelian tensor  category, and the above equivalence extends to an equivalence of abelian tensor categories $\Ind\langle X \rangle_{\otimes}\cong \Rep_k(G)$. Under this equivalence the unique extension of $\omega$ to $\Ind\langle X \rangle_{\otimes}$ corresponds to the forgetful functor on $\Rep_k(G)$ and is therefore a
faithful exact tensor functor.

We now show that $\Ind\langle X \rangle_{\otimes}$ embeds in $\mathcal C$ as a fully abelian tensor subcategory. For this we check that the assumptions of Lemma \ref{indlem} (2) are satisfied by the inclusion functor $\langle X \rangle_{\otimes}\to \mathcal C$. Firstly, small filtered colimits are exact in $\mathcal C$ as $\mathcal C$ is pointed by a faithful exact tensor functor $\vartheta:\,{\mathcal C}\to\Mod_R$ and they are exact in $\Mod_R$ (recall that we assumed that tensor functors commute with small colimits). Next, we verify that the map $\varinjlim \Hom_{\mathcal C}(P, X_i) \to \Hom_{\mathcal C}(P, \varinjlim X_i)$ is an isomorphism for $P$ and $X_i$ in $\langle X \rangle_{\otimes}$. For injectivity, we adapt the proof of (\cite{MR1012168}, Lemma 4.2.1 (ii)). For fixed $i$ the subobjects ${\rm Ker}(X_i\to X_j)$ for $j>i$ form an increasing system which must stabilize as $X_i\in \langle X \rangle_{\otimes}$ is noetherian. If $K_i$ is the common value, then the $X_i/K_i$ form an inductive system of subobjects of $X:=\varinjlim X_i$ whose colimit is still $X$. If a morphism $P\to X_i$ gives 0 when composed with $X_i\to X$, it must thus give 0 when composed with $X_i\to X_i/K_i$. But $X_i/K_i$ injects in $X_j$ for $j$ large enough, and we are done. For surjectivity, let $\phi:\, P\to X$ be a morphism, and denote by $Z$ its image. As in the proof of (\cite{MR1012168}, Lemma 4.2.2) we see that $Z\subset X_i/K_i$ for $i$ large enough, so as before $\phi$ comes from a morphism $P\to X_j$ for $j$ large enough.

\par We may thus apply the lemma and embed $\Ind\langle X \rangle_{\otimes}$ in $\mathcal C$ as claimed. Using Proposition \ref{proposition-regular-rep-is-p-v-ring} we then find a Picard-Vessiot ring $\mathcal{P}_{\omega}$ in $\mathcal{C}$ corresponding to $\omega$.
By construction, it satisfies all the required properties of Definition \ref{definition-p-v-ring}. Of these, only faithful flatness {\em in ${\mathcal C}$} requires further justification. It suffices to show that $\vartheta(P)$ is faithfully flat in $\Mod_R$, which in turn follows from \cite{delignemilne}, Theorem 3.2 (and its proof).
\end{proof}

\begin{corollary}\label{aut-galois}
The functor of automorphisms $\underline{\Aut}(\mathcal{P}{\mid}1)$ of a Picard--Vessiot ring $\mathcal P$ is representable by the Galois group scheme $G=\underline{\Aut}^\otimes(\omega_{\mathcal{P}})$.
\end{corollary}

In the case when the base field $k$ is algebraically closed, there is always a $k$-valued fibre functor on $\langle X \rangle_{\otimes}$ by \cite{MR1106898}, Corollaire 6.20. Thus the theorem implies:

\begin{corollary} If $k$ is algebraically closed, there exists a Picard--Vessiot ring for the subcategory $\langle X \rangle_{\otimes}$ in $\mathcal C$.
\end{corollary}

\begin{remark}\label{pvremas}\rm ${}$

\noindent (1) The general Tannakian theory (see e.g. \cite{delignemilne}, Theorem 3.2) tells us that the functor of isomorphisms $\underline{\Isom}^\otimes(\omega_{\mathcal{P}}\otimes_kR, \vartheta)$ is an affine $G$-torsor over $\Spec(R)$ and is represented by the spectrum of the faithfully flat $R$-algebra $\vartheta(P)$. The identity of $\vartheta(P)$ thus yields a canonical point of the torsor $\underline{\Isom}^\otimes(\omega_{\mathcal{P}}\otimes_kR, \vartheta)$.

\noindent (2) As Yves Andr\'e points out, one of the simplest situations where the above theory can be applied is the following. Consider the category ${\mathcal C}_0$ of triples $(V, W, \overline{\omega})$, where $V$ and $W$ are finite-dimensional $\Bbb Q$-vector spaces of the same dimension, and $\overline{\omega}$ is an isomorphism $V\otimes_{\Bbb Q}{\Bbb C}\stackrel\sim\to W\otimes_{\Bbb Q}{\Bbb C}$. This is a neutral Tannakian category over $\Bbb Q$ with fibre functor $\omega:\,(V, W, \overline{\omega})\mapsto W$. Its Ind-category $\mathcal C$ is equipped with a non-neutral fibre functor $\vartheta$ induced by $(V, W, \overline{\omega})\mapsto V$. We thus have a Picard--Vessiot theory for the restrictions of $\omega$ to subcategories of the form $\langle (V, W, \overline{\omega})\rangle_\otimes$ which in turn gives rise to a `motivic' theory in the following sense.

The category ${\mathcal C}_0$ is the target of the de Rham--Betti realization of motives modulo homological equivalence over $\Bbb Q$ (see \cite{andrebook}, 7.1.6 -- as explained there, one has to modify the commutativity constraint for the product on motives which involves a standard conjecture). The conjectured full faithfulness of the realization would imply that the motivic Galois group of a motive equals the Galois group scheme of its realization in ${\mathcal C}_0$.

\end{remark}

\section{Solution algebras}

As in the previous section, let $X$ be a dualizable object of the pointed tensor category $\mathcal{C}$. We assume that there exists a Picard--Vessiot ring $\mathcal P$ for $\langle X\rangle_\otimes$ in $\mathcal{C}$, and denote by $\omega:=\omega_{\mathcal{P}}$ the associated fibre functor.

Inspired by Andr\'e's definition of solution algebras for differential modules in \cite{MR3215927}, we put:

\begin{definition}\label{defi-solution-algebra}
A solution algebra for $\langle X \rangle_{\otimes}$ is a ring $\mathcal{S}$ in $\mathcal{C}$ such that
\begin{enumerate}
\item there exists an injective ring homomorphism $\iota \colon \mathcal{S} \to \mathcal{P}$ (i.e. this morphism is a monomorphism in $\mathcal{C}$),
\item there exists an object $Y$ of $\langle X \rangle_{\otimes}$ and a morphism $\sigma \colon Y \to \mathcal{S}$ in $\mathcal{C}$ such that the induced ring homomorphism $\Sym^*(Y) \to \mathcal{S}$ is surjective.
\end{enumerate}
\end{definition}

The equivalence of this definition with Andr\'e's in the case of differential modules in characteristic 0 will be proven in the more general context of ID-modules in Proposition \ref{proposition-sol-alg-equivalent-characterization}.

\begin{lemma} \hfill
\begin{enumerate}
\item Solution algebras are Ind-objects of $\langle X \rangle_{\otimes}$.
\item The extension of $\omega$ to the Ind-category $\Ind(\langle X \rangle_{\otimes})$ sends solution algebras to finitely generated $k$-algebras.
\item Given an embedding $\iota:\,\mathcal{S}\to \mathcal P$ as in the definition, the induced morphism $\Spec\omega(\mathcal{P})\to\Spec\omega(\mathcal{S})$ has schematically dense image.
\end{enumerate}
\end{lemma}

\begin{proof}
Property (2) of the definition of solution algebras gives statement (1); indeed, the symmetric algebra $\Sym^*(Y)$ is an ind-object of $\langle X \rangle_{\otimes}$, and the  Ind-category $\Ind(\langle X \rangle_{\otimes})$ is closed under subquotients in $\mathcal{C}$. If $\mathcal{S}$ and $Y$ are as in the definition, the $k$-vector space $\omega(Y)$ is finite dimensional and the morphism $\Sym^{*}(\omega(Y)) \to \omega(\mathcal{S})$ in the tensor category $\Mod(k)$ is surjective by exactness of $\omega$, whence statement (2). Finally, the injectivity of $\iota$ and the exactness of $\omega$ imply that the $k$-algebra homomorphism $\omega(\mathcal{S})\to\omega(\mathcal{P})$ is injective, whence (3).
\end{proof}

Note that $\Spec\omega(\mathcal{P})$ is nothing but the Tannakian fundamental group $G$ associated with $\mathcal P$. It is an affine group scheme of finite type over $k$ by (\cite{delignemilne}, Proposition 2.20 b)). Moreover, both $\Spec\omega(\mathcal{S})$ and $\Spec\omega(\mathcal{P})$ come equipped with a canonical $G$-action; the latter is just the usual (left) action of $G$ on itself.

We isolate these properties in a definition:

\begin{definition}\label{definiation-quasi-homogeneous-g-scheme}
Let $k$ be a field, and $G$ a group scheme of finite type over $k$. A quasi-homogeneous $G$-scheme over $k$ is a $G$-scheme $X$ of finite type over $k$ such that there exists a schematically dominant $G$-morphism $G\to X$, where $G$ is considered with its usual $G$-action.
\end{definition}

\begin{remark}
The image of the unit section of $G$ in $X$ gives a $k$-point of $X$ whose $G$-orbit $U$ is schematically dense in $X$. Since $G$ and $X$ are of finite type over $k$, the morphism $U\to X$ is an open immersion with schematically dense image by (\cite{MR0302656}, \S  III.3, Proposition 5.2) and (\cite{MR2675155}, Remark 10.31). It is necessarily the unique $G$-orbit on $X$ with these properties. When $k$ is of characteristic 0, both $G$ and $X$ are reduced (the latter by \cite{MR2675155}, Remark 10.32), and we recover the classical notion of quasi-homogeneous varieties used in \cite{MR3215927}. However, in the applications we shall also consider non-reduced $G$.
\end{remark}

Consider now the category of pairs $(\mathcal{S}, \iota)$, where  $\mathcal{S}$ is a solution algebra and $\iota:\, \mathcal{S}\hookrightarrow\mathcal{P}$ is the embedding specified in the definition. Morphisms of pairs are defined in the obvious way. By the preceding lemma and discussion, the functor $\Spec\circ\,\omega$ sends such a pair to an affine quasi-homogeneous $G$-scheme together with a distinguished $k$-point $z$ which is the image of the unit section of $G$ by the morphism $G\to\Spec\omega(\mathcal{P})$. In fact, we have the following direct generalization of (\cite{MR3215927}, Theorem 3.2.1):

\begin{theorem}\label{proposition-solution-algebra-equiavalence}
The assignment $(\mathcal{S}, \iota) \mapsto (\Spec(\omega(\mathcal{S})), z)$ gives an anti-equivalence between the above category of solution algebras and the category of affine quasi-homogeneous $G$-schemes of finite type over $k$ with a given $k$-point of the schematically dominant orbit.
\end{theorem}

\begin{proof} The composite functor $\Spec \circ \, \omega$ is fully faithful as it is the composition of fully faithful functors. For essential surjectivity let $Z$ be an affine quasi-homogeneous $G$-scheme of finite type over $k$ with a given $k$-point $z$ as above. By definition, we have a schematically dominant $G$-morphism $G \to Z$ sending the unit section to $z$. It corresponds to an injection  of $G$-algebras
$\mathcal{O}(Z) \hookrightarrow \mathcal{O}(G).$
As $\mathcal{O}(Z)$ is of finite type over $k$, we find a finite-dimensional $G$-invariant subspace $V\subset \mathcal{O}(Z)$ containing a system of $k$-algebra generators of $\mathcal{O}(Z)$. \hbox{Since} $\omega$ induces an equivalence of tensor categories between $\Ind(\langle X \rangle_{\otimes})$ and $\Rep_k(G)$, this morphism comes from a morphism $V\to\mathcal{S}$ in $\Ind(\langle X \rangle_{\otimes})$, with $V$ actually lying in $\langle X \rangle_{\otimes}$. As moreover $V$ contains a system of generators of $\mathcal{O}(Z)$, it gives rise to a surjection of $G$-algebras
$
\Sym^*(V) \twoheadrightarrow \mathcal{O}(Z)$
which translates back to property (2) of the definition of solution algebras via $\omega$. As for property (1), it  corresponds to the injection $\mathcal{O}(Z) \hookrightarrow \mathcal{O}(G)$ via $\omega$.
\end{proof}

\section{Iterative differential rings and modules}\label{ID}

In this section we apply the results of the previous one to the iterative differential modules of Matzat and van der Put \cite{MR1978401}. This theory has its origins in the concept of iterated differentials of Hasse--Schmidt, and is equivalent (in positive characteristic) to the theory of infinitely Frobenius-divisible modules of Katz.

Recall that an iterative differential ring (ID-ring for short) is a pair $\mathcal{R}=(R,\{ \partial_i \}_{i \geq 0})$, where $R$ is a commutative ring and $\partial_i \colon R \rightarrow R$ are additive maps for all $i \geq 0$ such that
\begin{enumerate}
\item $\partial_0=id_R$,
\item $\partial_i(r_1 r_2) = \sum_{j+j'=i} \partial_j(r_1) \partial_{j'}(r_2)$,
\item $\partial_i \circ \partial_j = \binom{i+j}{i} \partial_{i+j}$.
\end{enumerate}

\par The set $\{ \partial_j \}_{j \geq 0}$ of maps is called an iterative derivation on $\mathcal{R}$. The intersection $k$ of the kernels $\ker(\partial_j)$ is called the constant ring of $\mathcal{R}$.

\begin{remark}\rm  When $R$ is a ring containing the field $\mathbb{Q}$ of rational numbers,  every derivation $\delta$ on $R$ can be uniquely extended to an iterative derivation on $R$ by setting $\partial_i = \frac{1}{i!} \delta^i$. In particular, this is the case if $R$ is a simple differential ring of characteristic $0$. Thus in characteristic 0 the theory of ID-rings is equivalent to that of usual differential rings.
\end{remark}

\par Let $\mathcal{R}$ be an ID-ring and $I$ be an ideal in $R$. We say that $I$ is an iterative differential ideal (ID-ideal) if for all $i \geq 0$ we have $\partial_i(I) \subseteq I$. An iterative differential ring $\mathcal{R}$ is called simple if the only iterative differential ideals of $\mathcal{R}$ are $0$ and $R$.

\begin{proposition}\label{simple-id-ring}\hfill
\begin{enumerate}
\item If $\mathcal{R}$ is a simple ID-ring, then $R$ is an integral domain.
\item If $\mathcal{R}$ is a simple ID-ring with constant ring $k$ and $K$ is the fraction field of $R$, then there is a unique iterative derivation on $K$ extending the iterative derivation on $\mathcal{R}$. Moreover, the ring of constants of the iterative differential field $\mathcal{K}$ is $k$; in particular, $k$ is a field.
\end{enumerate}
\end{proposition}

\begin{proof} See \cite{MR1978401}, Lemma 3.2. \end{proof}

We now recall the definition of iterative connections. Let $\mathcal{R}$ be an ID-ring. An iterative differential module (or ID-module) $\mathcal{M}$ over $\mathcal{R}$ is a pair $(M,\{ \nabla_i \}_{i \geq 0})$, where $M$ is an $R$-module and $\nabla_i \colon M \rightarrow M$ are additive maps for $i \geq 0$ such that
\begin{enumerate}
\item $\nabla_0 = id_M$,
\item $\nabla_i(rm) = \sum_{j+j'=i} \partial_j(r) \nabla_{j'}(m)$,
\item $\nabla_i \circ \nabla_j = \binom{i+j}{i} \nabla_{i+j}$.
\end{enumerate}

\par The set of maps $\{ \nabla_i \}_{i \geq 0}$ is called an iterative connection on $M$ over $\mathcal{R}$.

ID-modules over a fixed $ID$-ring $\mathcal{R}$ form a tensor category with the tensor product and inner Hom operations defined as in (\cite{MR1978401}, Section 2.2). It becomes a pointed tensor category via the natural forgetful functor with values in $R$-modules. If the underlying module $M$ of an $ID$-module $\mathcal{M}$ is finitely generated and projective over $R$, then the inner Hom $\mathcal{M}^{\vee}=\mathcal{H}om_R(M,R)$ defines a dual for $M$ in the sense of symmetric monoidal categories.

The following proposition is a direct generalization of \cite{MR3215927}, Theorem 2.2.1 to the iterative differential setup. It is proven by exactly the same argument.

\begin{proposition}\label{proposition-fg-id-module-is-projective}
Let $\mathcal{R}$ a simple iterative differential ring and denote by $\mathcal{K}$ the quotient field of $\mathcal{R}$ with its canonical ID-structure. Let $\mathcal{M}$ be a finitely generated ID-module over $\mathcal{R}$, and set $\mathcal{M}_{\mathcal{K}}:=\mathcal{M}\otimes_{\mathcal{R}}\mathcal{K}$. We have the following:
\begin{enumerate}
\item The underlying module of $\mathcal{M}$ and its ID-subquotients are all projective modules.
\item The category consisting of objects that are ID-subquotients of finite direct sums of tensor products of the form $\mathcal{M}^{\otimes i} \otimes (\mathcal{M}^{\vee})^{j}$ form a rigid $k$-linear tensor category $\langle \mathcal{M} \rangle_{\otimes}$ over the constant field $k$ of $\mathcal{R}$.
\item The base change functor $\langle \mathcal{M} \rangle_{\otimes} \to \langle \mathcal{M}_{\mathcal{K}} \rangle_{\otimes}$ is an equivalence of $k$-linear tensor abelian categories.
\end{enumerate}
\end{proposition}

\par We can now define Picard-Vessiot rings for ID-modules by specializing Definition \ref{definition-p-v-ring} to the category of ID-modules over $\mathcal{R}$. In view of part (2) of the above proposition, Theorem \ref{pvfibfun} applies to the subcategory $\langle \mathcal{M} \rangle_{\otimes}$ and gives:

\begin{corollary} In the situation of the proposition there is an equivalence of categories between Picard--Vessiot rings for the subcategory $\langle \mathcal{M} \rangle_{\otimes}$   and neutral fibre functors on it. In particular, Picard--Vessiot rings exist if $k$ is algebraically closed.
\end{corollary}

\par Consider now the localization $\mathcal{M}_{\mathcal{K}}$. Combining Theorem \ref{pvfibfun} with Proposition \ref{proposition-fg-id-module-is-projective} (3), we obtain:

\begin{corollary}\label{pv-localization} The assignment $\mathcal P\mapsto\mathcal{P}_{\mathcal{K}}$ gives a bijective correspondence between Picard-Vessiot rings for $\langle \mathcal{M} \rangle_{\otimes}$ and $\langle \mathcal{M}_{\mathcal{K}} \rangle_{\otimes}$. Moreover, the associated Galois group schemes are naturally isomorphic.\end{corollary}

\par Next we turn to solution algebras, defined as in Definition \ref{defi-solution-algebra} in the special context of ID-modules. The following proposition shows that this notion is the exact analogue of Andr\'e's solution algebras (\cite{MR3215927} Definition 3.1.1) for ID-modules.

\begin{proposition}\label{proposition-sol-alg-equivalent-characterization}
An ID-ring $\mathcal{S}$ over $\mathcal{R}$ is a solution algebra for $\langle \mathcal{M} \rangle_{\otimes}$ if and only if it satisfies the following properties.
\begin{enumerate}
\item The underlying ring $S$ is an integral domain.
\item The constant field of the quotient field of $S$ is $k$.
\item There exists an object $\mathcal N$ in $\langle \mathcal{M} \rangle_{\otimes}$ and a  morphism $\mathcal{N} \to \mathcal{S}$ of ID-modules over $\mathcal{R}$ whose image generates $S$ as an $R$-algebra.
\end{enumerate}
\end{proposition}

\begin{proof}
Let first $\mathcal{S}$ be a solution algebra in the sense of Definition \ref{defi-solution-algebra}. We only have to check the first two conditions, as the third one is satisfied by definition. Property (1) follows as $\mathcal{S}$ can be embedded into a Picard-Vessiot ring $\mathcal P$ which is a simple ID-ring and hence an integral domain by Proposition \ref{simple-id-ring}. Moreover, since $\mathcal{R}$ is a simple ID-ring, the natural map $\mathcal{R} \to \mathcal{S}$ is injective, whence a chain $\mathcal{R} \subset \mathcal{S} \subset \mathcal{P}$ of ID-rings which are integral domains. Since the constant field of the quotient fields of $\mathcal{R} $ and $\mathcal{P}$ is $k$, the same is true for $\mathcal S$.

\par In the other direction, we only need to show that $\mathcal S$ embeds in a Picard--Vessiot ring of $\langle\mathcal{M}\rangle_\otimes$. This is proven by exactly the same argument as the characteristic zero case in (\cite{MR3215927}, Proposition 3.1.6 (1)).
\end{proof}

\par The general theory of solution algebras (Theorem \ref{proposition-solution-algebra-equiavalence}) gives:

\begin{corollary} There is an anti-equivalence between the category of solution algebras for $\langle \mathcal{M} \rangle_{\otimes}$ and the category of affine quasi-homogeneous $G$-schemes over $k$.\end{corollary}

\par Now consider again the category equivalence $\langle \mathcal{M}\rangle_{\otimes} \to \langle \mathcal{M}_{\mathcal{K}} \rangle_{\otimes}$ given by Proposition \ref{proposition-fg-id-module-is-projective} (3).

\begin{corollary} The above equivalence restricts to an equivalence between the full subcategories of solution algebras for $\langle \mathcal{M} \rangle_{\otimes}$ and $\langle \mathcal{M}_{\mathcal{K}} \rangle_{\otimes}$. The quasi-inverse assigns the intersection $\mathcal{S}' \cap \mathcal{P}$ to a solution algebra $\mathcal{S}'$ for $\langle \mathcal{M}_{\mathcal{K}} \rangle_{\otimes}$.
\end{corollary}

\begin{proof}
Given a Picard--Vessiot ring $\mathcal P$ for $\langle \mathcal{M} \rangle_{\otimes}$, a solution algebra $\mathcal{S}'$ for  $\langle \mathcal{M}_{\mathcal{K}} \rangle_{\otimes}$ is by definition a subring of $\mathcal P_{\mathcal K}$, and so is $\mathcal P$. A quasi-inverse to the base change functor $\mathcal{S}\mapsto \mathcal{S}_{\mathcal K}$ is thus given by $\mathcal{S}'\mapsto\mathcal{S}' \cap \mathcal{P}$.
\end{proof}

\par Consider now the ID-module $\mathcal{M}_{\mathcal{K}}$ over the ID-field $\mathcal{K}$.  We have the following generalization of Andr\'e's notion of solution fields for differential modules.

\begin{definition}
Let $\mathcal{L}{\mid}\mathcal{K}$ be an extension of ID-fields. We say that $\mathcal{L}$ is a {\em solution field} for $\langle \mathcal{M}_{\mathcal{K}} \rangle_{\otimes}$ if the constant field of $\mathcal{L}$ is $k$ and there exists an ID-module $\mathcal{N}_{\mathcal{K}}$ in $\langle \mathcal{M}_{\mathcal{K}} \rangle_{\otimes}$ and a morphism of ID-modules $\mathcal{N}_{\mathcal{K}} \to \mathcal{L}$ whose image generates the field extension $L{\mid}K$.
\end{definition}

\begin{proposition}\hfill

\begin{enumerate}
\item An ID-field extension $\mathcal{L}{\mid}\mathcal{K}$ is a solution field for $\langle \mathcal{M}_{\mathcal{K}} \rangle_{\otimes}$ if and only if it is the quotient field of a solution algebra $\mathcal{S}$ for $\langle \mathcal{M} \rangle_{\otimes}$.
\item Every solution field $\mathcal{L}$ for $\langle \mathcal{M}_{\mathcal{K}} \rangle_{\otimes}$ embeds as an intermediate ID-extension of $\mathcal{J}{\mid}\mathcal{K}$, where $\mathcal J$ is the quotient field of a Picard--Vessiot algebra for $\langle \mathcal{M}_{\mathcal{K}} \rangle_{\otimes}$.
\end{enumerate}
\end{proposition}

In accordance with the terminology of \cite{MR1978401}, we call an ID-field $\mathcal{J}$ as above a Picard--Vessiot field.

\begin{proof} Statement (1) follows from the Proposition \ref{proposition-sol-alg-equivalent-characterization} and the previous corollary. Statement (2) is a consequence of the definition of solution algebras.
\end{proof}

\par We next develop the Galois theory of solution fields. Given an ID-module $\mathcal M$ over $\mathcal R$, Corollaries \ref{aut-galois} and \ref{pv-localization} tell us that the Galois group scheme $G$  associated with a Picard--Vessiot ring $\mathcal P$ for $\mathcal M$ represents the $k$-group functor of ID-automorphisms of $\mathcal{P}_{\mathcal{K}}$ over $\mathcal{K}$. Using this representation we can naturally extend the action of $G$ to the fraction field $\mathcal{J}$  of $\mathcal{P}_{\mathcal{K}}$. An element $p/q \in \mathcal{J}$ is called invariant under a closed subgroup scheme $H$ of $G$ if for all $k$-algebras $k'$ and all $h \in H(k')$ we have an equality
\[
h(p \otimes 1) \cdot (q \otimes 1) = (p \otimes 1) \cdot h(q \otimes 1)\]
in $\mathcal{P}_{\mathcal{K}} \otimes_k k'$.
The set of invariant elements of $\mathcal{J}$ under $H$ is denoted by $\mathcal{J}^{H}$.
Recall now that there is the following iterative differential Galois correspondence:

\begin{theorem}[\cite{MR2657686}, Theorem 11.5]\label{galcor}
 The  map $H \mapsto \mathcal{J}^{H}$ gives an order-reversing bijection between  closed subgroup schemes $H$ of $G$ and intermediate ID-fields of $\mathcal{J}{\mid}\mathcal{K}$.
\end{theorem}

\par The above theorem is stated more generally for fields equipped with an iterable higher derivation in the reference, but in particular it applies to ID-modules as defined above. Note the important point that (in contrast to what is implicitly assumed in \cite{MR1978401}) in \cite{MR2657686} no separability assumption is made on Picard--Vessiot extensions.

\par In characteristic zero Andr\'e proved that solution fields correspond to observable subgroups of the Galois group. Here we need a slightly more general notion allowing non-reduced group schemes.
Namely, we call a closed subgroup scheme $H$ of an affine group scheme $G$ of finite type over a field $k$ {\em observable} if the quotient $G/H$ is quasi-affine over $k$.

We have the following equivalent characterizations of observable subgroups.

\begin{proposition}\label{theorem-observable-subgroup-equi}
Let $G$ be an affine group scheme of finite type over a field $k$, and let $H$ be a closed subgroup scheme of $G$. Then the following are equivalent:
\begin{enumerate}
\item The subgroup scheme $H$ is observable.
\item Every finite-dimensional $H$-representation is an $H$-subrepresentation of a finite dimensional $G$-representation.
\item There exists a finite dimensional $G$-representation $V$ and a vector $v \in V$ such that $H$ is the stabilizer subgroup scheme of the vector $v$ in $G$.
\end{enumerate}
\end{proposition}

\begin{proof}
The proof of equivalence $(1) \Leftrightarrow (2)$ is Theorem 1.3 in \cite{MR2891209}, and the proof of $(2) \Rightarrow (3)$ goes in the same way as the proof of implication $(7) \Rightarrow (2)$ in Theorem 2.1 of \cite{MR1489234}.
For the proof of $(3) \Rightarrow (1)$, we can use condition (3) together with (\cite{MR0302656}, \S III.3, Proposition 5.2) to construct an immersion $\iota$ of the quotient $G/H$ in the affine bundle $\mathbb{V}:=\Spec(\Sym^*(V))$. By (\cite{MR0302656}, \S I.2, Proposition 5.2), the image $\iota(G/H)$ is open in its closure, hence quasi-affine.
\end{proof}

We can now state the following generalization of (\cite{MR3215927}, Theorem 4.2.3 (3)) to iterative differential fields.

\begin{theorem}\label{idobs} In the situation of Theorem \ref{galcor} an intermediate ID-extension $\mathcal{L}$ of $\mathcal{J}{\mid}\mathcal{K}$ is a solution field for $\langle \mathcal{M}_{\mathcal{K}} \rangle_{\otimes}$ if and only if the corresponding subgroup scheme $H$ is an observable subgroup scheme of the Galois group scheme $G$.
\end{theorem}

\begin{proof}
Let $H$ be an observable subgroup scheme of $G$. By Proposition \ref{theorem-observable-subgroup-equi} (3) there exists a finite-dimensional $G$-representation $V$ and a vector $v \in V$ such that $H$ is the isotropy subgroup scheme of $v$ in $V$. Recall that $\mathcal{J}$ was defined as the quotient field of some Picard--Vessiot algebra $\mathcal{P}_{\mathcal{K}}$ for $\langle \mathcal{M}_{\mathcal{K}} \rangle_{\otimes}$, and denote by $\omega$ the associated fibre functor. Using the Tannakian equivalence induced by $\omega$, we can write $V$ as $\omega(\mathcal{N}_{\mathcal{K}}^{\vee})$ for some ID-module $\mathcal{N}_{\mathcal{K}}$ in $\langle \mathcal{M}_{\mathcal{K}} \rangle_{\otimes}$. By Remark \ref{rewrite} the vector $v$ determines an ID-homomorphism $v \colon \mathcal{N}_{\mathcal{K}} \to \mathcal{P}_{\mathcal{K}} \to \mathcal{J}$. Let $\mathcal{L}$ be the subfield of $\mathcal{J}$ generated by the image of this ID-homomorphism. By construction, $H$ is exactly the closed subgroup scheme of $G$ fixing $\mathcal{L}$ in its action on $\mathcal{J}$.

\par Conversely, if $\mathcal{L}$ is a solution field generated by a solution $v:\, \mathcal{N}_{\mathcal{K}}\to \mathcal{L}$, then the subgroup scheme $H$ attached to $\mathcal{L}$ by Theorem \ref{galcor} is the isotropy subgroup scheme of the solution $v$ in $\omega(\mathcal{N}_{\mathcal{K}}^{\vee})$, and hence $H$ is observable.
\end{proof}

\section{An example}

In this section we show that there are Picard--Vessiot extension of ID-fields giving rise to solution fields corresponding to non-reduced and non-normal subgroup schemes of the Galois group scheme.

Let $k$ be an algebraically closed field of prime characteristic $p \neq 2$. We define an iterative derivation on the polynomial ring $k[t]$ by setting $\partial_i(t^k) = \binom{k}{i} t^{k-i}$ and extending it linearly to the whole polynomial ring. Since $\partial_k(t^k) = 1$, we get that $k[t]$ is a simple ID-ring with this iterative derivation. Furthermore, the field of constants is precisely $k \subseteq k[t]$. The iterative derivation can be extended to the quotient field $k(t)$ of the polynomial ring $k[t]$.
\par Note that iterative derivations (resp. connections) are determined by the $p$-th power maps $\partial_{p^n}$ (resp. $\nabla_{p^n}$): if we write $n$ as the sum $a_0 + a_1 p + \ldots  + a_m p^m$, where $a_i \in \{0, 1, \ldots, p - 1 \}$, then
\[
(\partial_1)^{a_0} \circ (\partial_p)^{a_1} \circ \ldots \circ (\partial_{p^m})^{a_m} = c \cdot \partial_n,
\]
where $c$ is a non-zero element of $\mathbb{F}_p$. We now consider the following example.

\begin{example}\label{exidsol} \hfill
Let $\mathcal{M}$ be the ID-module corresponding to the sequence of equations
\[
\partial_{p^n}\begin{pmatrix} y_1 \\ y_2 \end{pmatrix} = \begin{pmatrix} 0 & 1 \\ 0 & a_n t^{-p^{n}} \end{pmatrix} \begin{pmatrix} y_1 \\ y_2 \end{pmatrix},
\]
where $a_n \in \{1, \ldots, p-1 \}$. Let $\textbf{s} = (s_1, s_2)^{T}$ be a non-trivial solution of this iterative differential equation. First, we see that $s_2$ is a solution of the iterative differential equation
\[
\partial_{p^n}(y) = a_n t^{-p^{n}} y,
\]
hence by (\cite{matzat-diff_gal_theory_pos_char}, Theorem 3.13) or (\cite{MR1978401}, Section 4) after a suitable choice of the coefficients $a_n$ we get that $s_2$ is transcendental over $k(t)$ and $\underline{Aut}^{\ID}(k[t][s_2]{\mid}k[t])=\mathbb{G}_m$. The multiplicative group $\mathbb{G}_m$ is thus a quotient of the differential Galois group
\[
G = \underline{Aut}^{\ID}(k[t][s_1, s_2]{\mid}k[t]).
\]
\par Moreover, for every element $h$ of the differential Galois group that fixes $s_2$, the element $h(s_1) - s_1$ is a constant, since
\[
\partial_{p^n}(h(s_1) - s_1) = h(\partial_{p^n}(s_1)) - \partial_{p^n}(s_1) = h(s_2) - s_2 = 0.
\]
This implies that $\underline{Aut}^{\ID}(k[t][s_1, s_2]{\mid}k[t][s_2])$ is a subgroup of the additive group $\mathbb{G}_a$. In conclusion, the differential Galois group is a closed subgroup scheme of $\mathbb{G}_m \ltimes \mathbb{G}_a$. The only subgroups of $\mathbb{G}_a$ that are stable under the action of the multiplicative group are the Frobenius kernels, but the differential Galois group must be reduced by (\cite{MR2657686} Corollary 11.7), thus $G$ is either $\mathbb{G}_m$ or $\mathbb{G}_m \ltimes \mathbb{G}_a$.
\par We now show that for a suitable choice of the $a_i$ the element $s_1$ is transcendental over $k(t)(s_2)$, which will imply that the differential Galois group is indeed the whole $\mathbb{G}_m \ltimes \mathbb{G}_a$. Assume this is not the case, and let $s_1 = b_0 + b_1 s_2 + b_2 s_2^2 + \ldots$ be a polynomial with coefficients in $k(t)$. Applying $\partial_1$, we get that
\begin{align*}
s_2 &= \partial_1(s_1) = \partial_1(b_0 + b_1 s_2 + b_2 s_2^2 + \ldots) = \\
&= \partial_1(b_0) + \partial_1(b_1)s_2 + b_1 a_0 t^{-1} s_2 + \textrm{second or higher order terms in }s_2.
\end{align*}
Therefore the element $b_1 \in k(t)$ must satisfy the  equation
\[
\partial_1(b_1) = 1 - a_0 t^{-1} b_1.
\]
Write $b_1$ as $f/g$, where $f, g \in k[t]$ are relatively prime polynomials. The previous equation can now be rewritten as
\begin{equation}
t \partial_1(f) g - t f \partial_1(g) = t g^2 - a_0 f g \label{equation-1}.
\end{equation}
Since $f$ and $g$ are relatively prime, the equation can hold only if $g$ divides $t \partial_1(g)$. This can happen either if $g = t \partial_1(g)$ or $\partial_1(g) = 0$. We shall obtain a contradiction by showing that neither case can happen for suitable $a_i$.
\par Let $g = t \partial_1(g)$: substituting back this identity to Equation (\ref{equation-1}) and dividing by $g$, we get that
\[
t \partial_1(f) - f = t g - a_0 f.
\]
If we assume that $a_0 \neq 1$, then we see that $t$ divides $f$, but it also divides $g$, which is a contradiction as $f$ and $g$ are relatively prime.
\par Let $\partial_1(g) = 0$: this implies that $g(t) = g_1(t^p)$ for some polynomial $g_1\in k[t]$. Again after substituting and dividing by $g$, equation (\ref{equation-1}) becomes
\[
t \partial_1(f) = t g - a_0 f.
\]
It follows that we may write $f=t \cdot f_1$ woth some $f_1\in k[t]$. Substituting and dividing by $t$ we obtain
\[
f_1 + t \partial_1(f_1) - g = - a_0 f_1.
\]
We now apply $\partial_1$ and use $\partial_1(g) = 0$ to obtain
\[
0 =(1 + a_0) \partial_1(f_1) + \partial_1(f_1) + 2 t \partial_2(f_1) =  (2 + a_0) \partial_1(f_1) + 2 t \partial_2(f_1).
\]
Since we assumed $p\neq 2$, the previous identity implies that the coefficients $c_i$ of $f_1$ must satisfy
\[
i \cdot c_i \cdot (a_0 + 2 + i - 1) = 0.
\]
If we set $a_0:=p - 1$, then we see that the  $c_i$ can only be nonzero for $p | i$, so that $f_1(t) = f_2(t^p)$ with some $f_2\in k[t]$. In summary, $b_1$ is of the form $t \frac{f_2(t^p)}{g_1(t^p)}$, but the first iterative derivative of such an element is $\frac{f_2(t^p)}{g_1(t^p)}$, not the required $1 + (p -1) t^{-1} t \frac{f_2(t^p)}{g_1(t^p)} = 1 + \frac{f_2(t^p)}{g_1(t^p)}$, a contradiction.
\par In conclusion: if  $a_0 = p - 1$ and the other $a_i$-s are chosen suitably, then the differential Galois group scheme is $\mathbb{G}_m \ltimes \mathbb{G}_a$. In this case we get solution fields corresponding to non-normal, non-reduced observable subgroup schemes, namely the $p^n$-th roots of unity $\mu_{p^n}$ in $\mathbb{G}_m \subseteq \mathbb{G}_m \ltimes \mathbb{G}_a$.

\end{example}

\section{Difference rings}

In this section we briefly explain how to apply the general theory of solution algebras in the context of difference Galois theory. In characteristic zero it would also be possible to treat difference modules as generalized (noncommutative) differential modules as in \cite{MR1862024} and then invoke the results of \cite{MR3215927} more or less directly, but the general theory of the present paper allows a quicker approach.

Let us first recall some basics. A difference ring $\mathcal{A}$ is a pair $(A, \sigma)$ where $A$ is a commutative ring and $\sigma \colon A \rightarrow A$ is a ring endomorphism.
A difference ideal of $\mathcal{A}$ is an ideal $I$ such that $\sigma(I) \subseteq I$. A simple difference ring is a difference ring with only the trivial difference ideals. Simple difference rings are always reduced, and their constant ring (i.e. the subring of fixed elements of $\sigma$) is a field (\cite{MR1480919}, Lemma 1.7).

\par A difference module $\mathcal{M}$ over $\mathcal{A}$ is a pair $(M,\Sigma)$ where $M$ is an $R$-module over $R$ and $\Sigma \colon M \to M$ is a $\sigma$-semilinear additive map. These form an abelian category with a natural tensor structure. However, inner Hom's do not exist in general, so one has to make some restrictions in order to dispose of them.

\par Assume moreover that $A$ is noetherian.
Under this assumption a difference module $\mathcal{M}=(M,\Sigma)$ is called {\em \'etale} if $M$ is finitely generated and the endomorphism $A_\sigma\otimes_AM\to M$ induced by $\lambda\otimes m\mapsto \lambda\cdot\Sigma(m)$ is bijective. (Here $A_\sigma$ denotes $A$ regarded as a module over itself via $\sigma$.) In case the endomorphism $\sigma:\, A\to A$ is flat, \'etale difference modules over $\mathcal A$ form an abelian tensor category having inner Hom's (see \cite{MR1106901}, A.1.15 and A.1.17). Thus, as in the case of differential modules, if moreover $M$ is projective over $A$, then $\mathcal{M}^\vee = {\mathcal Hom}(\mathcal{M}, \mathcal{A})$ defines a dual for $M$.

We then have the following analogue of Proposition \ref{proposition-fg-id-module-is-projective}, which is again proven by the same argument as \cite{MR3215927}, Theorem 2.2.1.

\begin{proposition}\label{proposition-fg-difference-module-over-simple-difference-ring-is-projective}
Let $\mathcal{A}=(A,\sigma)$ be a simple difference ring such that $A$ is a noetherian ring and $\sigma$ is flat, and let $\mathcal{M}$ be an  \'{e}tale difference module over $\mathcal{A}$. Denote by $k$ the constant field of $\mathcal{A}$, and by $T(\mathcal{A})$ its total ring of fractions. We have the following:
\begin{enumerate}
\item The underlying module of $\mathcal{M}$ and its difference subquotients are projective modules.
\item The category consisting of objects that are difference subquotients of finite direct sums of tensor products of the form $\mathcal{M}^{\otimes i} \otimes (\mathcal{M}^{\vee})^{\otimes j}$ form a rigid $k$-linear tensor category $\langle \mathcal{M} \rangle_{\otimes}$ over the constant field $k$ of $\mathcal{A}$.
\item The natural base change functor
    \[
    \langle \mathcal{M} \rangle_{\otimes} \to \langle \mathcal{M}_{T(\mathcal{A})} \rangle_{\otimes}
    \]
    is an equivalence of categories.
\end{enumerate}
\end{proposition}

\emph{From now on we assume we are in the situation of the proposition and moreover $k$ is algebraically closed.} As before, we consider the tensor category of \'etale difference modules over $\mathcal{A}$ as pointed via the natural forgetful functor $\vartheta$. By Theorem \ref{pvfibfun}  there exists a Picard-Vessiot ring $\mathcal{P}$ for $\mathcal{M}$ and the category $\langle \mathcal{M} \rangle_{\otimes}$ is equivalent to the category $\Repf_k(G)$, where $G$ is the Galois group scheme of $\mathcal{M}$ pointed at the Picard-Vessiot ring. We will denote by $\omega$ the fibre functor given by the Picard-Vessiot ring. Furthermore, Proposition \ref{proposition-fg-difference-module-over-simple-difference-ring-is-projective} (3) implies that the base change of the Picard-Vessiot ring to the total ring of fractions is the Picard-Vessiot ring of the difference module $\mathcal{M}_{T(\mathcal{A})}$, and the associated Galois group schemes are isomorphic.

We also note that the Picard-Vessiot ring $\mathcal{P}$ has the properties satisfied by the base ring $\mathcal{A}$: it is a noetherian ring (as $P$ is faithfully flat over $A$) and the endomorphism of $\mathcal{P}$ is the base change of $\sigma$ via $A\to P$ (as $\mathcal P$ is a colimit of \'{e}tale difference modules), hence flat.

The general definition of solution algebras applies in this context as well, and by Theorem \ref{proposition-solution-algebra-equiavalence} we obtain a correspondence between solution algebras and quasi-homogeneous schemes over the Galois group scheme.

We even have the following analogue of Proposition \ref{proposition-sol-alg-equivalent-characterization} for difference rings (proven in the same way):

\begin{proposition}
A difference ring $\mathcal{S}$ with flat endomorphism over $\mathcal{A}$ is a solution algebra for $\langle \mathcal{M} \rangle_{\otimes}$ if and only if the $S$ is contained in a noetherian simple difference ring with flat endomorphism whose constant field is $k$ and there exists a morphism $\mathcal{N} \to \mathcal{S}$ of \'{e}tale difference modules over $\mathcal{A}$ whose image generates $S$ as an $A$-algebra.
\end{proposition}

\begin{remark}\rm Plainly, the theory sketched above is not the most general possible. In particular, we imposed the condition that $k$ is algebraically closed only to ensure the existence of Picard--Vessiot rings (and also because we'll need this condition below). In situations where $k$ is non-closed but the existence of a neutral fibre functor is known, the above arguments work equally well. This is the case with the various categories of $t$-motives considered by Papanikolas in \cite{papa}. It would also be interesting to extend the abstract theory of solution algebras so that it covers the very general setting in which Ovchinnikov and Wibmer \cite{ow} define Picard--Vessiot rings for linear difference equations; it seems to us that this requires more substantial work.
\end{remark}

We now consider solution fields. Consider a difference field $\mathcal{K}$ (i.e. a field equipped with an endomorphism $\sigma$). {\em We assume throughout that $\sigma$ is bijective.} In this case an \'etale difference module $\mathcal{M}$ over $\mathcal{K}$ is just a difference module with bijective endomorphism $\Sigma$. The Picard--Vessiot ring $\mathcal{P}$ for $\mathcal{M}$ carries an endomorphism that is also bijective (being a direct limit of \'etale difference modules over $\mathcal K$). It extends to an automorphism of the total ring of fractions $T({P})$ which is a semisimple commutative ring, i.e. a finite product of fields. The resulting difference ring  $T(\mathcal{P})$ is called the total Picard-Vessiot ring of $\mathcal M$. We now define:

\begin{definition}
Let $\mathcal{L}{\mid}\mathcal{K}$ be an extension of difference rings. We say that $\mathcal{L}$ is a total solution ring for $\langle \mathcal{M} \rangle_{\otimes}$ if
\begin{enumerate}
\item every non-zerodivisor of $L$ is a unit in $L$,
\item the constant ring of $\mathcal{L}$ is $k$,
\item there exists a difference module $\mathcal{N}$ in $\langle \mathcal{M} \rangle_{\otimes}$ and a morphism of difference modules $\mathcal{N} \to \mathcal{L}$ such that the total fraction ring of the image of this homomorphism is $L$.
\end{enumerate}
\end{definition}

It follows from the definitions that total solution rings for $\langle \mathcal{M} \rangle_{\otimes}$ are exactly the total fraction rings of solution algebras. As the latter embed in the Picard-Vessiot algebra $\mathcal P$, total solution rings embed in $T(\mathcal{P})$.

We now quote the following Galois correspondence for total Picard-Vessiot rings in characteristic $0$ from \cite{MR1480919}, Theorem 1.29:

\begin{proposition}\label{proposition-difference-galois-correspondence}
Let $\mathcal{K}$ be a difference field of characteristic $0$ with bijective endomorphisn and algebraically closed constant field $k$, and let $\mathcal{M}$ be an \'etale  difference module over $\mathcal{K}$. Denote by $T(\mathcal{P})$ the total Picard-Vessiot ring of $\mathcal{M}$ over $\mathcal{K}$ and by $G$ the Galois group scheme. The maps $H \mapsto T(\mathcal{P})^H$ and $\mathcal{L} \mapsto \Aut(T(\mathcal{P}){\mid}\mathcal{L})$ define an order-reversing bijection between the set of closed subgroups of $G(k)$ and those intermediate difference rings of $\Aut(T(\mathcal{P}){\mid}\mathcal{K})$ where every non-zerodivisor is a unit.
\end{proposition}

Note that since $k$ is assumed to be algebraically closed of characteristic 0, one can work here with closed subgroups of $G(k)$ as in classical differential Galois theory.
Having this correspondence at our disposal, the same argument that proved Theorem \ref{idobs} also gives:

\begin{theorem}\label{solution-rings-difference-case}
Let $\mathcal{L}$ be an intermediate difference ring of $T(\mathcal{P}){\mid}\mathcal{K}$ in which every non-zerodivisor is a unit.

The ring $\mathcal{L}$ is a total solution ring for $\langle \mathcal{M} \rangle_{\otimes}$ if and only if the corresponding subgroup $H$ is an observable subgroup of the Galois group $G(k)$.
\end{theorem}

Finally, we discuss an application to transcendence theory pointed out to us by Yves Andr\'e. In (\cite{MR3215927}, Corollary 1.7.1) he explains how his results on solution algebras for differential modules imply a theorem of Beukers \cite{beukers} concerning the specialization of algebraic relations between E-functions. The theory of solution algebras for difference equations sketched above implies the analogous results of Adamczewski--Faverjon \cite{adamczewski} and Philippon \cite{philippon} for Mahler functions.

Recall that a $q$--Mahler system for an integer $q\geq 2$ is a system of functional equations
\begin{equation}\label{mahler}
\left(\begin{matrix}
f_1(z) \\
\vdots \\
f_n(z)
\end{matrix}\right)
= A(z)
\left(\begin{matrix}
f_1(z^q) \\
\vdots \\
f_n(z^q)
\end{matrix}\right)
\end{equation}
with $A(z)\in \GL_n({\overline{\Bbb Q}}(z))$ and $f_i(z)\in {\overline{\Bbb Q}}\{z\}$ for $i=1,\dots, n$. A function $f(z)\in {\overline{\Bbb Q}}\{z\}$ is a $q$-Mahler function if it is a component of a solution vector of a $q$-Mahler system. A complex number $a$ in the open unit disk is a singularity of the above Mahler system if $a^{q^r}$ is a pole of an entry of $A(z)$ or $A(z)^{-1}$ for some $r\geq 0$. Theorem 1.3 of \cite{philippon} states:

\begin{corollary}\label{mahlercor}
Assume $\alpha$ is a point of the pointed complex unit disk which is not a singularity of the system (\ref{mahler}). Then every polynomial relation among the values $f_1(\alpha),\dots, f_n(\alpha)$ with $\overline{\Bbb Q}$-coefficients is the specialization at $z=\alpha$ of a suitable polynomial relation among the functions $f_1(z),\dots, f_n(z)$ with $\overline{\Bbb Q}(z)$-coefficients.
\end{corollary}

The article \cite{adamczewski} contains another proof and a homogeneous version of this statement.

Let us indicate how to prove the corollary along the lines of (\cite{MR3215927}, Corollary 1.7.1). Consider the localization of $\overline{\Bbb Q}[z]$ by all linear polynomials $(z-a)$, where $a$ is a singularity of the system (\ref{mahler}), a $q^r$-th root of 1 or else 0, and equip it with the difference structure induced by $z\mapsto z^q$. The resulting difference ring $\mathcal{R}$ is noetherian and simple. To the system (\ref{mahler}) one associates a difference module over $\mathcal{R}$ which is \'etale under our assumptions on the numbers $a$. The functions $f_1(z),\dots, f_n(z)$ generate a solution algebra $\mathcal{S}$. Since algebraic Mahler functions are known to be rational, the field $\overline{\Bbb Q}(z)$ is algebraically closed in the fraction field of $S$. In other words, the morphism $\Spec(S)\to\Spec(R)$ has geometrically integral generic fibre. On the other hand, by Remark \ref{pvremas} (1) the $R$-algebra $S$ becomes isomorphic to $\omega(\mathcal{S})\otimes_kR$ after faithfully flat base change. This implies that all fibres of $\Spec(S)\to\Spec(R)$ are integral, in particular the one over $\alpha$. But they are of the same dimension by the analogue of the Siegel--Shidlovsky theorem proven by Nishioka \cite{nishioka}. This proves the corollary.

\bibliographystyle{abbrv}
\bibliography{nagypaper}

\end{document}